\documentclass[reqno]{amsart}
\usepackage{amssymb}
\usepackage{mathtools}
\usepackage{a4wide,amsmath}
\usepackage{mathrsfs}
\usepackage{amsthm}
\numberwithin{equation}{section}
\numberwithin{figure}{section}
\numberwithin{table}{section}
\usepackage{enumerate}
\usepackage[utf8]{inputenc}
\usepackage{hyperref}
\hypersetup{hidelinks}

\long\def\MSC#1\EndMSC{\def\arg{#1}\ifx\arg\empty\relax\else
	{\narrower\noindent%
		{2020 Mathematics Subject Classification}: #1\\} \fi}
\long\def\PACS#1\EndPACS{\def\arg{#1}\ifx\arg\empty\relax\else
	{\narrower\noindent%
		{PACS numbers}: #1}\fi}
\long\def\KEY#1\EndKEY{\def\arg{#1}\ifx\arg\empty\relax\else
	{\narrower\noindent% 
		Keywords: #1\\}\fi}

\theoremstyle{plain}
\newtheorem{theorem}{Theorem}[section]
\newtheorem{lemma}[theorem]{Lemma}
\newtheorem{proposition}[theorem]{Proposition}
\newtheorem{corollary}[theorem]{Corollary}
\theoremstyle{definition}

\theoremstyle{remark}
\newtheorem{remark}[theorem]{Remark}

\newcommand{\norm}[1]{\lVert#1\rVert}
\newcommand{\abs}[1]{\lvert#1\rvert} 
\newcommand{\inner}[1]{\langle#1\rangle} 
\newcommand{\mspan}{\mathop{\textup{span}}}
\newcommand{\essinf}{\mathop{\textup{ess\,inf}}}

\newcommand{\sgn}{\mathop{\textup{sgn}}}
\newcommand{\I}{\mathrm{i}}    % imaginary unit
\newcommand{\e}{\mathrm{e}}    % Euler's number
\newcommand{\di}{\mathrm{d}}   % differential

\newcommand{\R}{\mathbb{R}}
\newcommand{\N}{\mathbb{N}}
\newcommand{\C}{\mathbb{C}}
\newcommand{\Z}{\mathbb{Z}}
\newcommand{\redel}{\mathop{\textup{Re}}}

\begin{document}
	\title[Lipschitz stability for infinite-dimensional spaces of perturbations]{Linearised Calder\'on problem: Reconstruction and Lipschitz stability for infinite-dimensional spaces of unbounded perturbations}
	
	\author[H.~Garde]{Henrik Garde}
	\address[H.~Garde]{Department of Mathematics, Aarhus University, Ny Munkegade 118, 8000 Aarhus C, Denmark.}
	\email{garde@math.au.dk}
	
	\author[N.~Hyv\"onen]{Nuutti Hyv\"onen}
	\address[N.~Hyv\"onen]{Department of Mathematics and Systems Analysis, Aalto University, P.O. Box~11100, 00076 Helsinki, Finland.}
	\email{nuutti.hyvonen@aalto.fi}
	
	\begin{abstract}
		We investigate a linearised Calder\'on problem in a two-dimensional bounded simply connected $C^{1,\alpha}$ domain $\Omega$. After extending the linearised problem for $L^2(\Omega)$ perturbations, we orthogonally decompose $L^2(\Omega) = \oplus_{k=0}^\infty \mathcal{H}_k$ and prove Lipschitz stability on each of the infinite-dimensional $\mathcal{H}_k$ subspaces. In particular, $\mathcal{H}_0$ is the space of square-integrable harmonic perturbations. This appears to be the first Lipschitz stability result for infinite-dimensional spaces of perturbations in the context of the (linearised) Calder\'on problem. Previous optimal estimates with respect to the operator norm of the data map have been of the logarithmic-type in infinite-dimensional settings. The remarkable improvement is enabled by using the Hilbert--Schmidt norm for the Neumann-to-Dirichlet boundary map and its Fr\'echet derivative with respect to the conductivity coefficient. We also derive a direct reconstruction method that inductively yields the orthogonal projections of a general $L^2(\Omega)$ perturbation onto the $\mathcal{H}_k$ spaces, hence reconstructing any $L^2(\Omega)$ perturbation.
	\end{abstract}	
	\maketitle
	
	\KEY
	Calder\'on problem, 
	linearisation,
	Lipschitz stability,
	reconstruction.
	\EndKEY
	
	\MSC
	35R30, 35R25.
	\EndMSC
	
	\section{Introduction} \label{sec:intro}
	
	\subsection{Linearised Calder\'on problem} 
	
	Let $\Omega$ be a bounded simply connected $C^{1,\alpha}$ domain in $\mathbb{R}^2$ for some $\alpha\in (0,1)$. For an isotropic conductivity coefficient $\gamma\in L^\infty(\Omega;\mathbb{R})$, with $\essinf\gamma > 0$, and a Neumann boundary value (current density)
	\begin{equation*}
		f\in L^2_\diamond(\partial\Omega) = \{ g\in L^2(\partial\Omega) \mid \inner{g,1}_{L^2(\partial\Omega)} = 0 \},
	\end{equation*}
	the conductivity equation in $\Omega$ reads
	\begin{equation}
          \label{eq:cond_eq}
		-\nabla\cdot(\gamma\nabla u) = 0 \text{ in } \Omega, \qquad \nu\cdot(\gamma\nabla u) = f  \text{ on } \partial\Omega,
	\end{equation}
	where $\nu$ is the exterior unit normal of $\partial \Omega$. The Lax-Milgram lemma yields a unique weak solution 
	\begin{equation*}
		u_f^\gamma \in H_\diamond^1(\Omega) = \big\{ w\in H^1(\Omega) \mid \inner{w|_{\partial\Omega},1}_{L^2(\partial\Omega)}=0 \big\}
	\end{equation*}
    for \eqref{eq:cond_eq}, corresponding to the interior electric potential. The  Neumann-to-Dirichlet (ND) map $\Lambda(\gamma)f = u_f^\gamma|_{\partial\Omega}$ is a compact self-adjoint operator in $\mathscr{L}(L^2_\diamond(\partial\Omega))$, associating applied current densities with boundary voltage measurements. In two dimensions and under the assumed boundary regularity, the ND map actually belongs to $\mathscr{L}_{\textup{HS}}(L^2_\diamond(\partial\Omega))$, the space of Hilbert--Schmidt operators on $L^2_\diamond(\partial\Omega)$~\cite[Theorem~A.2]{Garde2022a}. The forward map $\gamma\mapsto\Lambda(\gamma)$ is Fr\'echet differentiable at $\gamma$ with respect to complex-valued $L^\infty(\Omega)$ perturbations, with the corresponding derivative  $D\!\Lambda(\gamma)$ in  $\mathscr{L}(L^\infty(\Omega),\mathscr{L}_{\textup{HS}}(L_\diamond^2(\partial\Omega))) \subset \mathscr{L}(L^\infty(\Omega),\mathscr{L}(L_\diamond^2(\partial\Omega)))$. 
	
	We investigate the Fr\'echet derivative at the unit conductivity,~i.e.~$F = D\!\Lambda(1)$. Let $u_f$ and $u_g$ be harmonic functions in $\Omega$ with $f$ and $g$, respectively, as their Neumann traces. Then $F$ satisfies
	\begin{equation} \label{eq:F}
		\inner{(F\eta)f,g}_{L^2(\partial \Omega)} = -\int_{\Omega} \eta\nabla u_f\cdot\overline{\nabla u_g}\,\di x, \qquad \eta\in L^\infty(\Omega),\quad  f,g\in L^2_\diamond(\partial\Omega),
	\end{equation}
    which we adopt as the definition of $F$ in the following. In our two-dimensional setting, one may actually consider more general (unbounded) $L^2(\Omega)$ perturbations.
	\begin{proposition} \label{prop:L2}
		$F$ extends to an operator in $\mathscr{L}(L^2(\Omega),\mathscr{L}(L_\diamond^2(\partial\Omega)))$ via \eqref{eq:F}.
	\end{proposition}
	The linearised Calder\'on problem consists of reconstructing a perturbation $\eta$ from the linearised data $F\eta$, in contrast to the nonlinear Calder\'on problem which seeks to reconstruct $\gamma$ from $\Lambda(\gamma)$. Calder\'on~\cite{Calderon1980} proved that the linearised forward map associated to the Dirichlet-to-Neumann (DN) map is injective for $L^\infty(\Omega)$ perturbations. Dos~Santos~Ferreira--Kenig--Sj\"ostand--Uhlmann~\cite{Ferreira2009} proved the corresponding result for partial data; see also~\cite{Sjostrand2016,Ferreira2020,Sharafutdinov2009} for related works. In fact, Calder\'on's original injectivity proof carries over to the case of general $L^2(\Omega)$ perturbations and the ND map, demonstrating that $F\eta = 0$ implies the Fourier transform of (the zero-extension of) $\eta$ vanishes. In section~\ref{sec:injectivity} we provide an alternative elementary proof for the injectivity of $F$ on $L^2(\Omega)$ based on polynomial approximation.

    Before presenting our two main results in the following subsection, let us showcase the main theorem on stability (Theorem~\ref{thm:main2}) by a conceptually simple special case: The Hilbert--Schmidt structure that $F$ exhibits on $L^\infty(\Omega)$ extends to harmonic perturbations in the closed subspace
	\begin{equation*}
		\mathcal{H}_0(\Omega) = \{ \eta\in L^2(\Omega) \mid \Delta\eta=0 \} \nsubseteq L^\infty(\Omega),
	\end{equation*}
	on which there is Lipschitz stability. 
	
	\begin{corollary} \label{coro:main2}
		$F\in\mathscr{L}(\mathcal{H}_0(\Omega),\mathscr{L}_\textup{HS}(L^2_\diamond(\partial\Omega)))$ and there exists $C>0$ such that
		\begin{equation*}
			\norm{\eta}_{L^2(\Omega)} \leq C\norm{F\eta}_{\mathscr{L}_{\textup{HS}}(L_\diamond^2(\partial\Omega))}
		\end{equation*}
        for all $\eta\in \mathcal{H}_0(\Omega)$.
	\end{corollary}
	
	\subsection{Main results}
	
	Let us next present the details of our main results, which are given on the unit disk $D$ in $\R^2\simeq \C$. The involved spaces and estimates can be transferred to $\Omega$ with the help of a Riemann mapping, as outlined in section~\ref{sec:conformal}. In particular, the constant in Corollary~\ref{coro:main2} satisfies
    \begin{equation} \label{eq:constant}
	    C \leq  \sqrt{\pi} \; \frac{\max_{z \in \partial\Omega} |\Psi'(z)| }{\min_{z \in \partial\Omega}|\Psi'(z)|}
    \end{equation}
    for any conformal mapping $\Psi$ of $\Omega$ onto $D$, with $\Psi'$ denoting its complex derivative. To simplify the presentation, we denote the analysed operator by $F$ for both domains $\Omega$ and $D$.
	
	Let $\psi_{j,k}$ for $j\in\Z$ and $k\in\N_0$ comprise the orthonormal Zernike polynomial basis for $L^2(D)$. In the polar coordinates $x = r\e^{\I\theta}\in D$,
	\begin{equation} \label{eq:Zernike}
		\psi_{j,k}(r\e^{\I\theta}) = \sqrt{\frac{\abs{j}+2k+1}{\pi}}R^{\abs{j}}_{\abs{j}+2k}(r)\e^{\I j\theta}, \qquad j\in\mathbb{Z},\enskip k\in\mathbb{N}_0,
	\end{equation}
	where
	\begin{equation*}
		R^{\abs{j}}_{\abs{j}+2k}(r) = \sum_{q=0}^k(-1)^q\binom{\abs{j}+2k-q}{q}\binom{\abs{j}+2k-2q}{k-q}r^{\abs{j}+2k-2q}.
	\end{equation*}
	We define a family of infinite-dimensional subspaces via
	\begin{equation*}
		\mathcal{H}_k = \overline{\mspan\{\psi_{j,k} \mid j\in\Z\}}, \qquad k \in \mathbb{N}_0,
	\end{equation*} 
	with the closure taken in $L^2(D)$, and set 
	\begin{equation*}
		\mathcal{W}_K = \bigoplus_{k=0}^K \mathcal{H}_k, \qquad K \in \mathbb{N}_0.
	\end{equation*}
	Since it is known that (cf.~Remark~\ref{remark:density})
	\begin{equation*}
		L^2(D) = \bigoplus_{k=0}^\infty \mathcal{H}_k,
	\end{equation*}
	$\{\mathcal{W}_K\}_{K=0}^\infty$ is an increasing family of closed subspaces approaching $L^2(D)$. Let $\eta\in L^2(D)$, i.e.\ there is an $\ell^2(\Z\times \N_0)$ sequence of coefficients $c_{j,k}$ such that
	\begin{equation} \label{eq:etaL2}
		\eta = \sum_{k=0}^\infty \underbrace{\sum_{j\in\Z} c_{j,k}\psi_{j,k}}_{\eta_k} = \sum_{k=0}^\infty \eta_k,
	\end{equation}
	where $\eta_k$ is the orthogonal projection of $\eta$ onto $\mathcal{H}_k$, defined via the coefficients $\{c_{j,k}\}_{j\in\Z}$. 
	
	Our first main result gives an exact reconstruction method for recovering $\eta\in L^2(D)$ from the linearised data $F\eta$. Moreover, if $\eta\in\mathcal{W}_K$, then the method only requires data associated to $2K+2$ Neumann boundary values. In the following theorem, $\sgn$ is the sign function with the convention $\sgn(0)=1$ and
	\begin{equation} \label{eq:fj}
		f_m(\e^{\I\theta}) = \frac{1}{\sqrt{2\pi}}\e^{\I m\theta}, \qquad m\in\Z,
	\end{equation}
    are the complex Fourier orthonormal basis functions for $L^2(\partial D)$.
	\begin{theorem} \label{thm:main1}
		Let $\eta\in L^2(D)$ be given as in \eqref{eq:etaL2}. Then
		\begin{align*}
			c_{j,k} &{}={} -\sqrt{\pi(\abs{j}+2k+1)} \binom{\abs{j}+2k}{k} \bigl\langle (F\eta)f_{\sgn(j)(k+1)},f_{\sgn(j)(\abs{j}+k+1)} \bigr\rangle_{L^2(\partial D)}\\[1mm]
			&\phantom{{}={}} -\sum_{q=0}^{k-1}c_{j,q} \frac{\sqrt{(\abs{j}+2k+1)(\abs{j}+2q+1)}}{\abs{j}+k+q+1}\binom{\abs{j}+2k}{k-q}
		\end{align*}
        for all $j\in\Z$ and $k\in\N_0$.
	\end{theorem}
	Let us carefully dissect how Theorem~\ref{thm:main1} can be implemented as a reconstruction method: To start with, $\eta_0$, i.e.~the coefficients $\{c_{j,0}\}_{j \in \Z}$, are reconstructed from the data $(F\eta)f_{\pm 1}$. Next, $\eta_1$,~i.e.~the coefficients $\{c_{j,1}\}_{j \in \Z}$, can be determined using the coefficients for $\eta_0$ and the data $(F\eta)f_{\pm 2}$. This process can be continued so that $\eta_k$ is reconstructed from $\eta_0,\dots,\eta_{k-1}$ and the data $(F\eta)f_{\pm(k+1)}$ for any $k\in\N_0$. Hence, the whole of $\eta$ can be inductively reconstructed one orthogonal component at a time. Due to the material in section~\ref{sec:conformal}, this reconstruction method also generalises for $\Omega$ via employing the aforementioned Riemann mapping $\Psi : \Omega\to D$. Indeed, if $\eta\in L^2(\Omega)$, one can define $\widetilde{\eta} = \eta \circ \Psi^{-1}$ in $L^2(D)$ and use \eqref{eq:transformed} in section~\ref{sec:conformal} to obtain
	\begin{equation*}
		\inner{(F\widetilde{\eta})f_{m},f_{n}}_{L^2(\partial D)} = \inner{(F\eta)\widehat{f}_{m},\widehat{f}_{n}}_{L^2(\partial \Omega)}
	\end{equation*}
	where $\widehat{f}_m = \abs{\Psi'}(f_m\circ\Psi)$. This demonstrates how the Neumann boundary values should be chosen on $\partial \Omega$ to reconstruct $\widetilde{\eta}$ via Theorem~\ref{thm:main1}. The original perturbation can subsequently be recovered as $\eta = \widetilde{\eta}\circ\Psi$.

    Next we shall focus on stability. If $\eta\in\mathcal{W}_K$, then \eqref{eq:etaL2} becomes
	\begin{equation} \label{eq:etasum}
		\eta = \sum_{j\in\Z}\sum_{k=0}^K c_{j,k}\psi_{j,k}.
	\end{equation}
	Based on the representation \eqref{eq:etasum}, we define a special {\em subset} of $\mathcal{W}_K$:
	\begin{equation} \label{eq:Aspace}
		\mathcal{A}_K = \bigl\{ \eta\in\mathcal{W}_K \mid \redel(c_{j,k}\overline{c_{j,k'}}) \geq 0 \text{ for } j\in\Z, \ k,k'\in\{0,\dots,K\} \bigr\}.
	\end{equation}
	Note that
	\begin{equation*}
		\bigcup_{k=0}^K \mathcal{H}_k \subseteq \mathcal{A}_K \subseteq \mathcal{W}_K.
	\end{equation*}
	Moreover, if the coefficients $c_{j,k}$ are real and the sign of $c_{j,k}$ (ignoring zero-coefficients) only depends on $j$, then $\eta\in\mathcal{A}_K$.
	
	Our second main result shows that the Hilbert--Schmidt structure for $F$ extends to perturbations in $\mathcal{W}_K$ and yields Lipschitz stability on $\mathcal{A}_K$. Below $H_K$ is the $K$'th harmonic number, with $H_0 = 0$.
	\begin{theorem} \label{thm:main2}
		For each $K\in\N_0$, $F\in \mathscr{L}(\mathcal{W}_K,\mathscr{L}_{\textup{HS}}(L^2_\diamond(\partial D)))$ with
		\begin{equation} \label{eq:Fbnd}
			\norm{F\eta}_{\mathscr{L}_{\textup{HS}}(L^2_\diamond(\partial D))} \leq \frac{2}{\sqrt{\pi}}\sqrt{2K+H_K+4} \,\norm{\eta}_{L^2(D)}, \qquad \eta\in \mathcal{W}_K.
		\end{equation}
		On $\mathcal{A}_K$ there is Lipschitz stability:
		\begin{equation} \label{eq:Fstability}
			 \norm{\eta}_{L^2(D)} \leq \sqrt{2\pi}\binom{2K}{K}\norm{F\eta}_{\mathscr{L}_{\textup{HS}}(L^2_\diamond(\partial D))}, \qquad \eta\in \mathcal{A}_K.
		\end{equation}
		For $K = 0$ the Lipschitz constant in \eqref{eq:Fstability} can be improved to $\sqrt{\pi}$.
	\end{theorem}
	
	In particular, Theorem~\ref{thm:main2} shows that on the infinite-dimensional $\mathcal{H}_k$ spaces,
	\begin{equation*}
		\norm{\eta}_{L^2(D)} \leq \sqrt{2\pi}\binom{2k}{k}\norm{F\eta}_{\mathscr{L}_{\textup{HS}}(L^2_\diamond(\partial D))}, \qquad \eta\in \mathcal{H}_k.
	\end{equation*}
	Notice that $\mathcal{H}_0$ is the space of square-integrable harmonic functions on $D$ since $\psi_{j,0}$, $j \in \Z$, are the standard harmonic basis functions for the unit disk. In consequence, $\mathcal{H}_0$ equals $\mathcal{H}_0(\Omega)$ up to a conformal mapping. Likewise, the $\mathcal{W}_K$ spaces can be transferred to $\Omega$ via a conformal mapping and, as outlined in section~\ref{sec:conformal}, this only modifies the constants in Theorem~\ref{thm:main2} by multiples that explicitly depend on the employed mapping; cf.~\eqref{eq:constant}. In particular, the $K$-dependence remains as in Theorem~\ref{thm:main2}. Moreover, the conformally transformed $\mathcal{H}_k$ spaces form an orthogonal decomposition of $L^2(\Omega)$ in a weighted inner product (equivalent to the standard inner product) that also depends explicitly on the utilised conformal mapping.  Stirling's approximation reveals the asymptotic behaviour of the Lipschitz constant in \eqref{eq:Fstability} for a large $K$:
	\begin{equation*}
		\sqrt{2\pi}\binom{2K}{K} \simeq \sqrt{\frac{2}{K}}\,4^{K},
	\end{equation*}
	which demonstrates an expected exponential growth in $K$.
        
    It should be mentioned that Zernike polynomials have previously been used in a numerical study on the stability of the linearised Calder\'on problem in \cite{Allers91}.
	
	\subsection{Comments on the main results}
	
	We emphasise that our results are inherently two-dimensional: The Sobolev embeddings allowing to define $F$ as an operator in $\mathscr{L}(L^2(\Omega),\mathscr{L}(L_\diamond^2(\partial\Omega)))$ via \eqref{eq:F} are sharp and do not hold in higher dimensions. Moreover, $F\eta$ is not a Hilbert--Schmidt operator in higher dimensions, not even for a bounded perturbation~\cite[Appendix~A]{Garde2022a}. Interestingly, the proofs for both parts of Theorem~\ref{thm:main2} can be reduced to known bounds on the trigamma function. 

	Let us then briefly discuss why \eqref{eq:Fbnd} does not seem to follow from previous results. In two dimensions, \cite[Theorem~A.2]{Garde2022a} shows that if $L\in \mathscr{L}(H^{-s}_\diamond(\partial\Omega),H_\diamond^s(\partial\Omega))$ for some $s > \tfrac{1}{4}$, then $L$ is in fact a Hilbert--Schmidt operator in $\mathscr{L}_{\textup{HS}}(L^2_\diamond(\partial\Omega))$, and there exists $C>0$ (independent of $L$) such that
	\begin{equation*}
		\norm{L}_{\mathscr{L}_{\textup{HS}}(L^2_\diamond(\partial\Omega))} \leq C\norm{L}_{\mathscr{L}(H^{-s}_\diamond(\partial\Omega),H_\diamond^s(\partial\Omega))}.
	\end{equation*}
	As usual, the $\diamond$-symbol refers to a zero-mean condition. For $s=\frac{1}{2}$ we thus have 
	\begin{equation*}
		F\in \mathscr{L}(L^\infty(\Omega),\mathscr{L}(H^{-1/2}_\diamond(\partial\Omega),H_\diamond^{1/2}(\partial\Omega))) \subset \mathscr{L}(L^\infty(\Omega),\mathscr{L}_{\textup{HS}}(L^2_\diamond(\partial\Omega))),
	\end{equation*}
	implying  Hilbert--Schmidt continuous dependence of $F\eta$ on $\eta\in L^\infty(\Omega)$. However, defining $F$ via \eqref{eq:F} for $\eta\in L^2(\Omega)$ requires the associated Neumann boundary values to be in $L^2_\diamond(\partial\Omega)$ in order to enable the use of suitable Sobolev embeddings. Therefore, one does not know a priori that $F\eta$ is a Hilbert--Schmidt operator for $\eta\in L^2(\Omega)$, but the Hilbert--Schmidt property and the associated continuous dependence must indeed be separately proved for $\eta \in \mathcal{W}_K$ to establish Theorem~\ref{thm:main2}. 
    
	To our knowledge, Theorem~\ref{thm:main2} is the first \emph{infinite-dimensional} Lipschitz stability result in the linearised Calder\'on problem, which nicely complements the long tradition for Lipschitz stability results in \emph{finite-dimensional} settings for the nonlinear Calder\'on problem; see~e.g.~\cite{Harrach_2019,Alberti2019,Alberti2020,Alberti2022} for some recent general results. This suggests that it may be beneficial to consider stability in terms of ND maps, instead of DN maps, due to the more desirable topological properties of the former. It is also worth mentioning that our proofs do not rely on the standard arguments involving complex geometric optics solutions. Indeed, it seems to be the Hilbert space structures of $L^2(\Omega)$ perturbations and Hilbert--Schmidt operators that enable deducing Lipschitz stability for the infinite-dimensional $\mathcal{H}_k$ spaces, in comparison to the optimal logarithmic bounds in terms of operator norms in the infinite-dimensional linearised and nonlinear Calder\'on problems~\cite{Alessandrini1988,Alessandrini2012,Mandache2001,Koch2021,Nagayasu_2009,Li2022,Kow2022}. 
	
	In the related problem of detection and shape reconstruction of inclusions/obstacles, Lipschitz stability has recently been proved for general classes of polygonal and polyhedral inclusions \cite{Beretta2020,Beretta2021,Beretta2022}; for inclusions of small volume fraction this has already been known for some time~\cite{Friedman1989b}. Moreover, the injectivity for certain classes of unbounded coefficients, strictly contained in $L^2(\Omega)$, has been shown in the nonlinear Calder\'on problem in two dimensions~\cite{Astala2016,Carstea2016,Nachman2020}. The exact reconstruction of inclusions defined by $A_2$-Muckenhoupt coefficients has also recently been proved in an arbitrary spatial dimension $d\geq 2$~\cite{Garde2022b}.  
	
	The standard spaces $L^\infty(\Omega)$ and $\mathscr{L}(L^2_\diamond(\partial\Omega))$ involved in the linearisation of the Calder\'on problem, with the ND map as the data, are non-separable Banach spaces, with dual spaces that are poorly characterised for the use in practical computations where adjoints are needed. In contrast, resorting to $L^2(\Omega)$ perturbations and the Hilbert--Schmidt operator topology allows a simple characterisation for the adjoint of $F$. To demonstrate this, let $\widetilde{\mathcal{W}}_K$ be the conformally transformed $\mathcal{W}_K$ onto the domain $\Omega$. Both $\widetilde{\mathcal{W}}_K$ and $\mathscr{L}_{\textup{HS}}(L^2_\diamond(\partial\Omega))$ are separable Hilbert spaces with explicit orthonormal bases. Let $P_K$ be the orthogonal projection of $L^2(\Omega)$ onto $\widetilde{\mathcal{W}}_K$, let $\{g_m\}_{m\in\mathcal{I}}$ be an orthonormal basis for $L^2_\diamond(\partial\Omega)$, and set $u_m = u_{g_m}$ in \eqref{eq:F}. A simple calculation reveals that the adjoint $F^* : \mathscr{L}_{\textup{HS}}(L^2_\diamond(\partial\Omega)) \to\widetilde{\mathcal{W}}_K$ is given by
	\begin{equation} \label{eq:Fadj}
		F^*G = -\sum_{m,n\in\mathcal{I}} \inner{Gg_m,g_n}_{L^2(\partial\Omega)} P_K(\overline{\nabla u_m}\cdot \nabla u_n).
	\end{equation}
	Hence, even if the perturbation $\eta$ is bounded, e.g.\ it belongs to one of the spaces $\widetilde{W}_K\cap L^\infty(\Omega)$, the Hilbert--Schmidt topology enables practical optimisation methods for reconstruction. Note also that Lipschitz stability leads to favourable convergence rates for iterative algorithms; see~e.g.~\cite{Hoop2012}. 
	
	Due to the bound in \eqref{eq:Fstability},  $F|_{\mathcal{H}_k}$ has a closed range for the closed subspace $\mathcal{H}_k$. This is a desirable property for the series reversion approximations introduced in~\cite{Garde2022a}, although it remains an open problem to adapt the method of~\cite{Garde2022a} for $L^2(\Omega)$ perturbations. Be that as it may, the original motivation for this work was the question of whether there exists an infinite-dimensional space of perturbations satisfying the assumptions in \cite{Garde2022a}. 
	
	\subsection{Article structure} This article is organised as follows. Section~\ref{sec:conformal} outlines how the problem setting can be reduced to the unit disk domain $D$. Sections~\ref{sec:proofprop},~\ref{sec:proofmain2},~and~\ref{sec:reconstruction} present the proofs for Proposition~\ref{prop:L2}, Theorem~\ref{thm:main2}, and Theorem~\ref{thm:main1}, respectively. In section~\ref{sec:injectivity} we present a simple alternative proof for the injectivity of $F$ on $L^2(\Omega)$. 
	
	\section{Reduction to the unit disk} \label{sec:conformal}
	
	In this section, we demonstrate that it is sufficient to replace $\Omega$ by $D$ in the forthcoming proofs by virtue of the Riemann mapping theorem. To this end, let $\Psi : \Omega \to D$ be a Riemann mapping with the inverse $\Phi = \Psi^{-1}$. The complex derivatives of $\Psi$ and $\Phi$, when interpreted as mappings between subsets of $\C$, are denoted by $\Psi'$ and $\Phi'$. Due to the assumed $C^{1,\alpha}$ boundary regularity for $\Omega$, the Kellogg--Warschawski theorem (see~\cite[Theorem~3.6 \& Exercise~3.3.5]{Pommerenke1992} and \cite[Theorem~12]{Warschawski1932}) ensures that $\Psi$ and $\Phi$ have $C^{1,\alpha}$ extensions, with non-vanishing first derivatives, to the closures of the respective domains. In particular, $\Psi'$ and $\Phi'$ are bounded on $\overline{\Omega}$ and $\overline{D}$, respectively.  The absolute value $\abs{\Phi'}$ restricted to $\partial D$ is the Jacobian determinant of the boundary transformation $\Phi : \partial D\to \partial \Omega$, while $\abs{\Phi'}^2$ is the Jacobian determinant for the domain transformation $\Phi : D\to \Omega$ by virtue of the Cauchy--Riemann equations. The analogous conclusions naturally apply to $|\Psi'|$ as well.

	The quantities defined on $\Omega$ in \eqref{eq:F} are mapped via $\widetilde\eta = \eta\circ\Phi$ and $\widetilde{u}_{\widetilde{f}} = u_f\circ\Phi$ to functions on~$D$. The conformal invariance of \eqref{eq:F} implies
	\begin{equation} \label{eq:transformed}
		\int_\Omega \eta\nabla u_f\cdot \overline{\nabla u_g}\, \di x = \int_{D} \widetilde{\eta}\,\nabla \widetilde{u}_{\widetilde{f}}\cdot \overline{\nabla \widetilde{u}_{\widetilde{g}}}\, \di x, 
	\end{equation}
    which means that if either of the two integrals exists, then the same also applies to the other integral with the same value; section~\ref{sec:proofprop} below demonstrates that the right-hand side of \eqref{eq:transformed} is well defined for $\eta \in L^2(\Omega)$ and $f,g \in L^2_\diamond(\partial \Omega)$. Moreover, $\widetilde{u}_{\widetilde{f}}\in H^1(D)/\mathbb{C}$ admits a characterisation as a harmonic function on $D$ with the Neumann boundary value 
	\begin{equation*}
		\widetilde{f} = \abs{\Phi'}(f\circ\Phi)
	\end{equation*}
	in $L_\diamond^2(\partial D)$; see~e.g.~\cite[Lemma~4.1 and Remark~4.1]{Hyvonen_2018} based on the $C^{1,\alpha}$ boundary regularity. Obviously, one can choose $\widetilde{u}_{\widetilde{f}}$ and $\widetilde{u}_{\widetilde{g}}$ to be elements of $H_\diamond^1(D)$ in \eqref{eq:transformed}. 
		
    Let $\norm{\Psi'}_{L^\infty(\partial\Omega)}$ and $\norm{\Phi'}_{L^\infty(\partial D)}$ denote norms of $\abs{\Psi'}$ and $\abs{\Phi'}$ on the domain boundaries. The maximum modulus principle entails 
    \begin{equation*}
    	\norm{\Psi'}_{L^\infty(\Omega)} = \norm{\Psi'}_{L^\infty(\partial\Omega)} \qquad \text{and} \qquad \norm{\Phi'}_{L^\infty(D)} = \norm{\Phi'}_{L^\infty(\partial D)}.
    \end{equation*}
	Thus we have
	\begin{equation} \label{eq:eta_trans}
		\frac{1}{\norm{\Psi'}_{L^\infty(\partial\Omega)}}\norm{\widetilde{\eta}}_{L^2(D)}\leq\norm{\eta}_{L^2(\Omega)} \leq \norm{\Phi'}_{L^\infty(\partial D)}\norm{\widetilde{\eta}}_{L^2(D)},
	\end{equation}
	as well as
    \begin{equation} \label{eq:f_trans}
	    \frac{1}{\norm{\Phi'}_{L^\infty(\partial D)}^{1/2}}\norm{\widetilde{f}}_{L^2(\partial D)}\leq\norm{f}_{L^2(\partial \Omega)} \leq \norm{\Psi'}_{L^\infty(\partial \Omega)}^{1/2}\norm{\widetilde{f}}_{L^2(\partial D)}.
    \end{equation}
	The Hilbert--Schmidt norms 
	\begin{equation*}
		\norm{F\eta}_{\mathscr{L}_{\textup{HS}}(L^2_\diamond(\partial\Omega))} \qquad \text{and} \qquad \norm{F\widetilde{\eta}}_{\mathscr{L}_{\textup{HS}}(L^2_\diamond(\partial D))}
	\end{equation*}
	are also equivalent (including the case of both being infinite). Recall that we denote the studied operator by the same symbol $F$ on both domains for simplicity. As this last equivalence requires a bit of work, we present a brief proof in what follows.
	
	Let $\{f_m\}_{m\in\Z}$ be the standard complex Fourier orthonormal basis for $L^2(\partial D)$ given in \eqref{eq:fj}, meaning that $\{f_m\}_{m\in\Z\setminus\{0\}}$ is an orthonormal basis for $L^2_\diamond(\partial D)$. Setting $\widehat{f}_m = \abs{\Psi'}(f_m\circ\Psi)$ defines an orthonormal basis $\{\widehat{f}_m\}_{m\in\Z}$  for $L^2(\partial \Omega)$ in the weighted $L^2$ inner product with the weight $\abs{\Psi'}^{-1}$. Consequently, $\{\abs{\Psi'}^{-1/2}\widehat{f}_m\}_{m\in\Z}$ is an orthonormal basis for $L^2(\partial\Omega)$ in the standard $L^2$ inner product.
	
	Let $P$ be the orthogonal projection of $L^2(\partial\Omega)$ onto $L^2_\diamond(\partial\Omega)$. By using \eqref{eq:transformed} and \eqref{eq:F}, we obtain
	\begin{align*}
		\norm{F\widetilde{\eta}}_{\mathscr{L}_{\textup{HS}}(L^2_\diamond(\partial D))}^2 &= \sum_{m,n \in \Z\setminus\{0\}}\big|\inner{(F\widetilde{\eta})f_m,f_n}_{L^2(\partial D)}\big|^2 = \sum_{m,n \in \Z\setminus\{0\}}\big|\inner{(F\eta)\widehat{f}_m,\widehat{f}_n}_{L^2(\partial\Omega)} \big|^2 \\
		&\leq \sum_{m,n \in \Z}\big|\big\langle\abs{\Psi'}^{1/2}(F\eta)P\abs{\Psi'}^{1/2}\abs{\Psi'}^{-1/2}\widehat{f}_m,\abs{\Psi'}^{-1/2}\widehat{f}_n \big\rangle_{L^2(\partial\Omega)}\big|^2 \\
		&= \big\| \abs{\Psi'}^{1/2}(F\eta)P\abs{\Psi'}^{1/2} \big\|_{\mathscr{L}_{\textup{HS}}(L^2(\partial\Omega))}^2 \\[1mm]
		&\leq \big\| \abs{\Psi'}^{1/2} \big\|_{\mathscr{L}(L^2(\partial\Omega))}^4 \norm{(F\eta)P}_{\mathscr{L}_{\textup{HS}}(L^2(\partial\Omega))}^2,
	\end{align*}
	where $\abs{\Psi'}^{1/2}$ is interpreted as a multiplication operator, and in the final step a standard norm inequality for products of operators is used; see~e.g.~\cite[Theorem~7.8(c)]{Weidmann1980}. Since $P$ projects onto the orthogonal complement of the constant function,
	\begin{equation*}
		\norm{(F\eta)P}_{\mathscr{L}_{\textup{HS}}(L^2(\partial\Omega))} = \norm{F\eta}_{\mathscr{L}_{\textup{HS}}(L_\diamond^2(\partial\Omega))}.
	\end{equation*}
	By also performing the analogous calculation in the opposite direction, we finally deduce
	\begin{equation} \label{eq:HSnormequiv}
	\frac{1}{\norm{\Psi'}_{L^\infty(\partial\Omega)}}\norm{F\widetilde{\eta}}_{\mathscr{L}_{\textup{HS}}(L_\diamond^2(\partial D))}	\leq \norm{F\eta}_{\mathscr{L}_{\textup{HS}}(L^2_\diamond(\partial\Omega))} \leq \norm{\Phi'}_{L^\infty(\partial D)}\norm{F\widetilde{\eta}}_{\mathscr{L}_{\textup{HS}}(L_\diamond^2(\partial D))},
	\end{equation}
    where the identities $\norm{\abs{\Psi'}^{1/2}}_{\mathscr{L}(L^2(\partial\Omega))} = \norm{\Psi'}_{L^\infty(\partial\Omega)}^{1/2}$ and  $\norm{\abs{\Phi'}^{1/2}}_{\mathscr{L}(L^2(\partial D))} = \norm{\Phi'}_{L^\infty(\partial D)}^{1/2}$ have also been utilised.
	
	\section{Proof of Proposition~\ref{prop:L2}} \label{sec:proofprop}
	
	Based on \eqref{eq:transformed}, \eqref{eq:eta_trans}, and \eqref{eq:f_trans} in section~\ref{sec:conformal}, it is sufficient to prove the result on $D$. Let $\eta\in L^2(D)$ and let $u_f$ and $u_g$ be harmonic functions on $D$ with the Neumann boundary values $f,g\in L^2_\diamond(\partial D)$, respectively. According to \cite[Chapter~2, Remark~7.2]{Lions1972}, $u_f,u_g\in H^{3/2}(D)/\C$ with continuous dependence on the Neumann data,~i.e.
	\begin{equation} \label{eq:ufregular}
		\norm{u_f}_{H^{3/2}(D)/\C} \leq C\norm{f}_{L^2(\partial D)} \qquad \text{and} \qquad \norm{u_g}_{H^{3/2}(D)/\C} \leq C\norm{g}_{L^2(\partial D)}.
	\end{equation}
	Using the continuous embedding $H^{1/2}(D)\hookrightarrow L^4(D)$ (e.g.~\cite[Corollary~4.53]{Demengel2012}), we may estimate as follows:
	\begin{align*}
		\big| \inner{(F\eta)f,g}_{L^2(\partial D)} \big| &= \Bigl|\int_D \eta\nabla u_f\cdot\overline{\nabla u_g}\,\di x\Bigr| \\[1mm]
		&\leq \norm{\eta}_{L^2(D)}\norm{\nabla u_f}_{L^4(D)}\norm{\nabla u_g}_{L^4(D)} \\[1mm]
		&\leq C \norm{\eta}_{L^2(D)}\norm{\nabla u_f}_{H^{1/2}(D)}\norm{\nabla u_g}_{H^{1/2}(D)} \\[1mm]
		&\leq C \norm{\eta}_{L^2(D)}\norm{u_f}_{H^{3/2}(D)/\mathbb{C}}\norm{u_g}_{H^{3/2}(D)/\mathbb{C}}.
	\end{align*} 
	Combining this with \eqref{eq:ufregular} concludes the proof.  \hfill\qed

	\begin{remark}
		Denote by $u_f^\gamma$ the solution to \eqref{eq:cond_eq} for a (nonconstant) coefficient $\gamma$ and with $\Omega$ replaced by $D$. The mapping
	    \begin{equation*}
      		N_\gamma: H_\diamond^{s}(\partial D) \ni f \mapsto  u_f^\gamma \in H^{s+3/2}(D)/\C
      	\end{equation*}
      	is bounded for $s = \pm 1/2$ if $\gamma \in C^{0,1}(\overline{D}; \R)$; cf.~e.g.~\cite[Corollary 2.2.2.6]{Grisvard1985}, which can be modified for the quotient space. It thus follows from complex interpolation theory for Hilbert spaces (e.g.~\cite[Chapter~1, Theorems~5.1, 7.7, 9.6, 13.2]{Lions1972}) that the solution map $N_\gamma$ remains bounded for $s = 0$, i.e.~from $L_\diamond^{2}(\partial D)$ to  $H^{3/2}(D)/\C$, for $\gamma \in C^{0,1}(\overline{D}; \R)$. By the above analysis, this demonstrates that the Fr\'echet derivative $D\!\Lambda(\gamma)$ in the unit disk $D$ belongs to $\mathscr{L}(L^2(D), \mathscr{L}(L^2_\diamond(\partial D)))$ for a Lipschitz continuous coefficient $\gamma$. Based on the material in section~\ref{sec:conformal}, this result also extends for a $C^{1,\alpha}$ domain $\Omega$ with a Lipschitz continuous conductivity coefficient.
	\end{remark}
	
	\section{Proof of Theorem~\ref{thm:main2}} \label{sec:proofmain2}
	
	To shorten the notation, we denote  the standard $L^2(\partial D)$ inner product by $\inner{\,\cdot\,,\,\cdot\,}$ and the Hilbert--Schmidt norm on $\mathscr{L}_\textup{HS}(L^2_\diamond(\partial D))$ by $\norm{\,\cdot\,}_{\textup{HS}}$ in the following.
	
	\subsection{On the Zernike polynomials}

    Recall the Zernike polynomials $\psi_{j,k}$, $j \in \Z$, $k \in \N_0$, given as functions of the polar coordinates $x = r\e^{\I\theta}\in D$ in \eqref{eq:Zernike}. On the weighted $L^2((0,1))$ space 
	\begin{equation*}
		L^2_r((0,1)) = \biggl\{ f:(0,1)\to\mathbb{C} \text{ measurable} \ \Big | \  \int_0^1 \abs{f(r)}^2 r\,\di r < \infty \biggr\}
	\end{equation*}
	with the inner product
	\begin{equation*}
		\inner{f,g}_{L^2_r((0,1))} = \int_0^1 f(r)\overline{g(r)}r\,\di r,
	\end{equation*}
	the radial Zernike polynomials corresponding to the same angular frequency $j$ are orthogonal (see~e.g.~the original paper by Zernike \cite{Zernike1934}),
	\begin{equation} \label{eq:Rorthogonal}
		\big \langle R^{\abs{j}}_{\abs{j}+2k},R^{\abs{j}}_{\abs{j}+2k'} \big\rangle_{L^2_r((0,1))} = \frac{\delta_{k,k'}}{2\abs{j}+4k+2}, \qquad k,k' \in \N_0.
	\end{equation}
    This immediately leads to the orthonormality of $\{ \psi_{j,k} \}_{j \in \Z, k \in \N_0}$; we have included a comment on the density of the span of the Zernike polynomials in $L^2(D)$ as Remark~\ref{remark:density} below, because this matter is not often explicitly discussed in the literature.
        
    Monomials can be expanded as
	\begin{equation} \label{eq:rpower}
		r^{\abs{j}+2p} = \sum_{s=0}^p d_{\abs{j},s,p} R_{\abs{j}+2s}^{\abs{j}}(r), \qquad j \in \Z, \ p \in \N_0,
	\end{equation}
	where
	\begin{equation} \label{eq:d_coefficient}
		d_{\abs{j},s,p} = \frac{\abs{j}+2s+1}{\abs{j}+p+s+1}\binom{p}{s}\binom{\abs{j}+p+s}{s}^{-1} = \frac{(\abs{j}+2s+1)(p)_s}{(\abs{j}+p+s+1)(\abs{j}+p+s)_{s}}
	\end{equation}
	in terms of the Pochhammer symbols defined as $(p)_s = \frac{p!}{(p-s)!}$; see~e.g.~\cite[Equation~(35)]{Janssen2004}.
	
	\begin{remark} \label{remark:density}
		The density of 
		\begin{equation*}
			\mspan\{\psi_{j,k} \mid j\in\Z, \  k\in\N_0\}
		\end{equation*}
		in $L^2(D)$ is typically not explicitly discussed, and since the indices for the radial and angular parts are not independent, this might not be immediately obvious. We briefly argue this fact, ensuring that $\{\psi_{j,k}\}_{j\in\Z,k\in\N_0}$ is indeed an orthonormal basis for $L^2(D)$. Consider $h(re^{\I\theta}) = r^m \e^{\I n\theta}$ in polar coordinates for some $n\in\Z$ and $m\in\N_0$. Since $r^{\abs{n}+2p}\e^{\I n\theta}\in \mspan\{\psi_{n,k}\mid k\in\N_0\}$ for $p\in\N_0$ by \eqref{eq:rpower}, the approximation of $h$ in $L^2(D)$ by the Zernike polynomials can be straightforwardly reduced to approximating $r^{m}$ in $L^2((0,1))$ by polynomials in $\mathcal{P}_{\abs{n}} = \mspan\{r^{\abs{n}+2p} \mid p\in\N_0\}$. However, $\mathcal{P}_{\abs{n}}$ is dense in $L^2((0,1))$ by Lemma~\ref{lemma:polapprox} in section~\ref{sec:injectivity}. As
		\begin{equation*}
			\mspan\{r^m\e^{\I n\theta} \mid n\in\Z, \enskip m\in\N_0\}
		\end{equation*}
		is dense in $L^2(D)$, the argument is complete.
	\end{remark}
	
	\subsection{Infinite matrix coefficients for $F$}
	
	Recall that $\{f_m\}_{m\in\Z\setminus\{0\}}$ is an orthonormal basis for $L^2_\diamond(\partial D)$ with $f_m$ defined in \eqref{eq:fj}. Set $u_m = u_{f_m}$ and $u_n = u_{f_n}$ in \eqref{eq:F} with $\Omega$ replaced by $D$, and let us continue to work in the polar coordinates $x = r\e^{\I \theta}$. Since 
	\begin{equation*}
          u_m( r\e^{\I \theta}) = \frac{1}{\sqrt{2 \pi}|m|} r^{|m|} \e^{\I m\theta} \qquad {\rm and} \qquad 
		\nabla u_m = U_\theta \begin{pmatrix}
			\partial_r u_m \\
			r^{-1}\partial_\theta u_m
		\end{pmatrix},
	\end{equation*}
	where $U_\theta$ is a certain unitary matrix, a direct calculation reveals
	\begin{align} 
		\nabla u_m\cdot\overline{\nabla u_n} &= \partial_r u_m\overline{\partial_r u_n} + r^{-2}\partial_\theta u_m\overline{\partial_\theta u_n} \notag \\[1mm] 
		&= \frac{(\abs{mn}+mn)}{2\pi\abs{mn}} r^{\abs{m}+\abs{n}-2}\e^{\I(m-n)\theta}  \notag \\[1mm]
		&= \begin{dcases}
			\tfrac{1}{\pi}r^{\abs{m}+\abs{n}-2}\e^{\I (m-n)\theta} & \text{for } mn>0, \\[1mm]
			0 & \text{for } mn< 0.
		\end{dcases} \label{eq:Fip}
	\end{align}
    It is thus enough to concentrate on the case $mn>0$ in what follows.

	Let $\eta\in L^2(D)$, i.e.~there is an $\ell^2(\Z\times\N_0)$ sequence of coefficients $c_{j,k}$ such that
	\begin{equation*}
		\eta(r\e^{\I\theta}) = \sum_{j\in\Z}\sum_{k=0}^\infty c_{j,k} \psi_{j,k}(r\e^{\I\theta}) = \sum_{j\in\Z}\sum_{k=0}^\infty c_{j,k}\sqrt{\frac{\abs{j}+2k+1}{\pi}}R^{\abs{j}}_{\abs{j}+2k}(r)\e^{\I j\theta}.
	\end{equation*} 
    For $mn>0$ it follows from \eqref{eq:F} and \eqref{eq:Fip} that
	\begin{equation*}
		\inner{(F\eta)f_m,f_n} = -\frac{2}{\sqrt{\pi}}\sum_{k=0}^\infty c_{n-m,k}\sqrt{\abs{n-m}+2k+1} \, \big\langle R^{\abs{n-m}}_{\abs{n-m}+2k},r^{\abs{m}+\abs{n}-2}\big\rangle_{L^2_r((0,1))}.
	\end{equation*}
	Furthermore, if $mn>0$, then $\abs{m}+\abs{n}-2 = \abs{n-m}+2v_{m,n}$ where $v_{m,n} = \min\{\abs{m},\abs{n}\}-1$. Hence, \eqref{eq:rpower} gives
	\begin{equation*}
		r^{\abs{m}+\abs{n}-2} = r^{\abs{n-m}+2v_{m,n}} = \sum_{s=0}^{v_{m,n}} d_{\abs{n-m},s,v_{m,n}} R^{\abs{n-m}}_{\abs{n-m}+2s}(r), \qquad mn>0,
	\end{equation*}
	and an application of \eqref{eq:Rorthogonal} results in (cf.~\cite[Equation (3.11)]{Wunsche05})
	\begin{align*}
		\inner{(F\eta)f_m,f_n} &= -\frac{1}{\sqrt{\pi}}\sum_{k=0}^{v_{m,n}} c_{n-m,k}\frac{d_{\abs{n-m},k,v_{m,n}}}{\sqrt{\abs{n-m}+2k+1}}, \qquad mn>0.
	\end{align*}
	Substituting $n = m+j$ with some $j\in\Z$ yields
	\begin{align*} \label{eq:nonzero_coef}
		\inner{(F\eta)f_m,f_{m+j}} &= -\frac{1}{\sqrt{\pi}}\sum_{k=0}^{v_{m,m+j}} c_{j,k}\frac{d_{\abs{j},k,v_{m,m+j}}}{\sqrt{\abs{j}+2k+1}}, \qquad m(m+j)>0.
	\end{align*}
	Inserting \eqref{eq:d_coefficient} finally gives the sought-for infinite matrix coefficients
	\begin{equation} \label{eq:Fetacomp}
		\inner{(F\eta)f_m,f_{m+j}} = \begin{dcases}
			-\frac{1}{\sqrt{\pi}}\sum_{k=0}^{v_{m,m+j}} c_{j,k}\frac{\sqrt{\abs{j}+2k+1}}{\abs{j}+v_{m,m+j}+k+1}p_{m,j,k} & \text{for } m(m+j)>0, \\[1mm]
			0 & \text{for } m(m+j)< 0,
		\end{dcases}
	\end{equation}
	with 
	\begin{equation*}
		p_{m,j,k} = \frac{(v_{m,m+j})_k}{(\abs{j}+v_{m,m+j}+k)_{k}} = \begin{dcases}
		\prod_{s=1}^k \frac{v_{m,m+j}+1-s}{\abs{j}+v_{m,m+j}+k+1-s} & \text{for } k\in\N, \\[1mm]
		1 & \text{for } k = 0.	
		\end{dcases}
	\end{equation*}
	
	\subsection{Boundedness of $F$ on $\mathcal{W}_K$ in the Hilbert--Schmidt norm}
	
	We proceed to prove \eqref{eq:Fbnd}. Assume that $\eta\in \mathcal{W}_K$ for some $K\in\N_0$, i.e.~$c_{j,k} = 0$ whenever $k > K$. Because obviously
	\begin{equation*}
		p_{m,j,k} \leq 1,
	\end{equation*}
 	we may estimate for $m(m+j)>0$ as follows:
	\begin{align*}
		\big| \inner{(F\eta)f_m,f_{m+j}} \big|^2 &\leq \frac{1}{\pi}\biggl(\sum_{k=0}^{K} \abs{c_{j,k}}\frac{\sqrt{\abs{j}+2k+1}}{\abs{j}+v_{m,m+j}+k+1}\biggr)^2 \\
		&\leq \frac{4}{\pi}\Bigl(\sum_{k=0}^{K} \abs{c_{j,k}}^2\Bigr)\sum_{k=0}^K\frac{\abs{j}+2k+1}{(2\abs{j}+2v_{m,m+j}+2k+2)^2}.
	\end{align*}
	For $m(m+j)>0$ it also holds $2v_{m,m+j} = \abs{m}+\abs{m+j}-2 -\abs{j}$ and $\abs{m+j}\geq 1$, yielding
	\begin{equation} \label{eq:Fcomp1}
		\big| \inner{(F\eta)f_m,f_{m+j}} \big|^2 \leq \frac{4}{\pi}\Bigl(\sum_{k=0}^{K} \abs{c_{j,k}}^2\Bigr)\sum_{k=0}^K\frac{\abs{j}+2k+1}{(\abs{m}+\abs{j}+2k+1)^2}, \qquad m(m+j)>0.
	\end{equation}

	In the following we will need bounds on the trigamma function, see~e.g.~\cite[Lemma~1]{Guo2013}: 
	\begin{equation} \label{eq:trigamma}
		\frac{2x+1}{2x^2} < \sum_{q=0}^\infty\frac{1}{(q+x)^2} < \frac{x+1}{x^2}, \qquad x\in(0,\infty).
	\end{equation}
	This leads to
	\begin{align}
		\sum_{m\in\Z\setminus\{0\}}\sum_{k=0}^K\frac{\abs{j}+2k+1}{(\abs{m}+\abs{j}+2k+1)^2} &\leq 2\sum_{k=0}^K (\abs{j}+2k+1)\sum_{m=0}^\infty\frac{1}{(m+\abs{j}+2k+1)^2} \notag\\
		&< 2\sum_{k=0}^K (\abs{j}+2k+1)\frac{\abs{j}+2k+2}{(\abs{j}+2k+1)^2} \notag\\
		&\leq 2\sum_{k=0}^K\frac{2k+2}{2k+1} = 2\Bigl(K+2 + \sum_{k=1}^K\frac{1}{2k+1}\Bigr) \notag\\[1mm]
		&\leq 2K+H_K+4, \label{eq:sumupbnd}
	\end{align}
	where $H_K$ is the $K$'th harmonic number, i.e.\ the $K$'th partial sum of the harmonic series, with $H_0 = 0$. Combining \eqref{eq:Fetacomp}, \eqref{eq:Fcomp1}, and \eqref{eq:sumupbnd}, we finally have
	\begin{align*}
		\norm{F\eta}_{\textup{HS}}^2 &= \sum_{m(m+j)>0} \big| \inner{(F\eta)f_m,f_{m+j}} \big|^2 \\
		&\leq \frac{4}{\pi}(2K+H_K+4)\sum_{j\in\Z}\sum_{k=0}^{K} \abs{c_{j,k}}^2 \\
		&= \frac{4}{\pi}(2K+H_K+4)\norm{\eta}_{L^2(D)}^2
	\end{align*}
    which proves \eqref{eq:Fbnd}.
	
	\subsection{Lipschitz stability on $\mathcal{A}_K$}
	
	The next objective is to prove \eqref{eq:Fstability}. We start by focusing on $\psi_{j,k}$ with fixed $j\in\Z$ and $k\in\N_0$. According to \eqref{eq:Fetacomp}, the only nonzero matrix elements $\inner{(F\psi_{j,k})f_m,f_{n}}$ correspond to $n = m+j$ and $m\in\mathcal{I}_{j,k}$ with 
	\begin{equation*}
		\mathcal{I}_{j,k} = \{m \in \Z \mid m(m+j)>0, \ \abs{m}>k, \ \abs{m+j}>k\}.
	\end{equation*}
	Using $2v_{m,m+j} = \abs{m}+\abs{m+j}-2 -\abs{j}$ for $m\in\mathcal{I}_{j,k}$, we deduce
	\begin{equation} \label{eq:Fcomp2}
		\inner{(F\psi_{j,k})f_m,f_{n}} = \begin{dcases}
			-\frac{2}{\sqrt{\pi}}\frac{\sqrt{\abs{j}+2k+1}}{\abs{j}+\abs{m}+\abs{m+j}+2k}\, p_{m,j,k} & \text{for } n = m+j,\enskip m\in\mathcal{I}_{j,k}, \\[1mm]
			0 & \text{else}.
		\end{dcases}
	\end{equation}
	Moreover, for $m\in\mathcal{I}_{j,k}$ and $k\in\N$,
	\begin{equation*}
		p_{m,j,k} = \prod_{s=1}^k\frac{\abs{m}+\abs{m+j}-\abs{j}-2s}{\abs{j}+\abs{m}+\abs{m+j}+2k-2s};
	\end{equation*}
    remember that $p_{m,j,0}=1$ for any $m$ and $j$.

	Recall that $\sgn:\R\to\{-1,1\}$ is the sign function with $\sgn(0) = 1$. It is straightforward to check that the specific indices
	\begin{equation} \label{eq:mindex}
		m_{q,j,k} = \sgn(j)(\abs{j}+q+k+1)
	\end{equation}
	belong to $\mathcal{I}_{j,k}$ for any $q\in\N_0$. In particular,
    \begin{equation*}
	    \abs{m_{q,j,k}} = \abs{j}+q+k+1 \qquad {\rm and} \qquad \abs{m_{q,j,k} + j} = 2\abs{j}+q+k+1. 
    \end{equation*}
    Hence,
	\begin{align} \label{eq:pochmjk}
		p_{m_{q,j,k},j,k} &= \prod_{s=1}^k\frac{\abs{j}+q+k+1-s}{2\abs{j}+q+2k+1-s} \geq \prod_{s=1}^k\min\Bigl\{\frac{1}{2},\frac{q+k+1-s}{q+2k+1-s}\Bigr\} \notag\\
		&\geq \prod_{s=1}^k\min\Bigl\{\frac{1}{2},\frac{k+1-s}{2k+1-s}\Bigr\} = \binom{2k}{k}^{-1}.
	\end{align}
	The first inequality follows from each of the fractions being either increasing or decreasing in $\abs{j}$, with the choice between the two options only depending on the sign of $s-1-q$. The second inequality is due to each of the considered fractions being increasing in $q$. The last equality results from each of the fractions being decreasing in $s$ and equaling $\frac{1}{2}$ for $s = 1$. Note that this lower bound also holds for $k=0$, as it yields the optimal value $1$.
	
	Combining \eqref{eq:Fcomp2}, \eqref{eq:mindex}, and \eqref{eq:pochmjk}, we can determine a lower bound for $\norm{F\psi_{j,k}}_{\textup{HS}}$:
	\begin{align}
		\norm{F\psi_{j,k}}_{\textup{HS}}^2 &= \sum_{m\in\mathcal{I}_{j,k}} \big| \inner{(F\psi_{j,k})f_m,f_{m+j}} \big|^2 \geq \sum_{q = 0}^\infty \big| \inner{(F\psi_{j,k})f_{m_{q,j,k}},f_{m_{q,j,k}+j}} \big|^2 \notag\\
		&\geq \frac{1}{\pi}\binom{2k}{k}^{-2}(\abs{j}+2k+1)\sum_{q=0}^\infty\frac{1}{(2\abs{j}+2k+1+q)^2} \notag\\
		&> \frac{1}{\pi}\binom{2k}{k}^{-2}(\abs{j}+2k+1)\frac{4\abs{j}+4k+3}{2(2\abs{j}+2k+1)^2} \notag\\[1mm]
		&> \frac{1}{2\pi}\binom{2k}{k}^{-2}, \label{eq:Fpsibnd}
	\end{align}
	where the penultimate inequality corresponds to the lower bound in \eqref{eq:trigamma}.
	
	\begin{remark} \label{remark:k_is_0}
		In the case $k = 0$, we could have defined $m_{q,j,0} = \sgn(j)(q+1)$ without the $\abs{j}$-term. The additional $\abs{j}$-term in \eqref{eq:mindex} is only needed for establishing a $j$-independent lower bound on $p_{m_{q,j,k},j,k}$ in \eqref{eq:pochmjk}, which is not relevant for $k=0$ as $p_{m,j,0} = 1$ for any $m$ and $j$. Following the same line of reasoning as above, this observation leads to replacing $\frac{1}{2\pi}$ by  $\frac{1}{\pi}$ in \eqref{eq:Fpsibnd}, which gives a more accurate lower bound for the case $k = 0$.
	\end{remark}
	
	Assume now $\eta\in \mathcal{A}_K$. In particular, there is an $\ell^2$ sequence of coefficients $c_{j,k}$ such that
	\begin{equation*}
		\eta = \sum_{j\in\Z}\underbrace{\sum_{k=0}^Kc_{j,k}\psi_{j,k}}_{\varphi_j} = \sum_{j\in\Z}\varphi_j.
	\end{equation*}
	In \eqref{eq:Fcomp2} all the terms $\inner{(F\psi_{j,k})f_m,f_n}$ are real and nonpositive. Thus, based on the condition $\redel(c_{j,k}\overline{c_{j,k'}})\geq 0$ in \eqref{eq:Aspace}, we have the inequality
	\begin{equation*}
		\abs{\inner{(F\varphi_j)f_m,f_n}}^2 = \biggl\lvert \sum_{k=0}^K c_{j,k} \inner{(F\psi_{j,k})f_m,f_n} \biggr\rvert^2 \geq \sum_{k=0}^K \abs{c_{j,k}}^2 \big| \inner{(F\psi_{j,k})f_m,f_n} \big|^2.
	\end{equation*}
	Thereby \eqref{eq:Fpsibnd} also gives a lower bound for $\norm{F\varphi_{j}}_{\textup{HS}}$:
	\begin{equation} \label{eq:varphibnd}
		\norm{F\varphi_{j}}_{\textup{HS}}^2 \geq \frac{1}{2\pi}\sum_{k=0}^K \abs{c_{j,k}}^2 \binom{2k}{k}^{-2} \geq \frac{1}{2\pi}\binom{2K}{K}^{-2}\sum_{k=0}^K \abs{c_{j,k}}^2.
	\end{equation}
	
	According to \eqref{eq:Fcomp2}, the nonzero elements of the infinite matrix $\{\inner{(F \varphi_j)f_m,f_n}\}_{m,n\in\Z\setminus\{0\}}$ lie on its $j$'th diagonal. Hence, $\{F\varphi_j\}_{j\in\Z}$ is an orthogonal set in the Hilbert--Schmidt inner product. Due to this orthogonality and \eqref{eq:varphibnd}, we have
	\begin{equation*}
		\norm{F\eta}_{\textup{HS}}^2 = \Bigl\lVert\sum_{j\in\Z} F\varphi_j\Bigr\rVert_{\textup{HS}}^2 = \sum_{j\in \Z}\norm{F\varphi_j}_{\textup{HS}}^2 \geq \frac{1}{2\pi}\binom{2K}{K}^{-2}\sum_{j\in\Z}\sum_{k=0}^K \abs{c_{j,k}}^2 = \frac{1}{2\pi}\binom{2K}{K}^{-2}\norm{\eta}_{L^2(D)}^2.
	\end{equation*} 
	Together with Remark~\ref{remark:k_is_0}, this shows \eqref{eq:Fstability} and concludes the proof of Theorem~\ref{thm:main2}. \hfill\qed
	
	\section{Proof of Theorem~\ref{thm:main1}} \label{sec:reconstruction}
	
	Let $\eta\in L^2(D)$, i.e.\ there is an $\ell^2(\Z\times \N_0)$ sequence of coefficients $c_{j,k}$ such that
	\begin{equation*}
		\eta = \sum_{k=0}^\infty \underbrace{\sum_{j\in\Z} c_{j,k}\psi_{j,k}}_{\eta_k} = \sum_{k=0}^\infty \eta_k.
	\end{equation*}
	In other words, $\eta_k$ is the orthogonal projection of $\eta$ onto $\mathcal{H}_k$ defined via the coefficients $\{c_{j,k}\}_{j\in\Z}$. Obviously,
    \begin{equation} \label{eq:eta_minusk}
		\kappa_k = \eta - \sum_{q=0}^{k-1} \eta_q = \sum_{q=k}^\infty \sum_{j\in\Z} c_{j,q}\psi_{j,q}, \qquad k \in \N_0,
	\end{equation}
    does not include any $c_{j,q}$ coefficients for $q<k$.

    Recall the infinite matrix coefficients from \eqref{eq:Fetacomp} for $m(m+j)>0$,
	\begin{equation} \label{eq:Fetacomprecon}
		\inner{(F\eta)f_m,f_{m+j}}_{L^2(\partial D)} = -\frac{1}{\sqrt{\pi}}\sum_{q=0}^{v_{m,m+j}} c_{j,q}\frac{\sqrt{\abs{j}+2q+1}\,(v_{m,m+j})_q}{(\abs{j}+v_{m,m+j}+q+1)(\abs{j}+v_{m,m+j}+q)_{q}}.
	\end{equation}
	We still denote by $\sgn$ the sign function satisfying $\sgn(0) = 1$ and make use of the special indices
	\begin{equation*}
		m_{j,k} = \sgn(j)(k+1), \qquad j\in\Z,\enskip k\in\N_0,
	\end{equation*}
	which leads to
	\begin{equation*}
		v_{m_{j,k},m_{j,k}+j} = \min\{\abs{m_{j,k}},\abs{m_{j,k}+j}\}-1 = k
	\end{equation*}
	for all $j\in\Z$. Substituting $\kappa_k$ from \eqref{eq:eta_minusk} for $\eta$ and $m_{j,k}$ for $m$ in \eqref{eq:Fetacomprecon} removes all but the last term in the sum over $q$, with that last term corresponding to $q = k$. This gives
	\begin{equation*}
		\Bigl\langle F\Bigl(\eta-\sum_{q=0}^{k-1}\eta_q\Bigr)f_{m_{j,k}},f_{m_{j,k}+j} \Big\rangle_{L^2(\partial D)} = \, -\frac{1}{\sqrt{\pi}}\, c_{j,k}\frac{(k)_k}{\sqrt{\abs{j}+2k+1}\,(\abs{j}+2k)_{k}},
	\end{equation*} 
	or equivalently,
	\begin{equation} \label{eq:cjkrecon}
		c_{j,k} = -\sqrt{\pi(\abs{j}+2k+1)}\,\frac{(\abs{j}+2k)_{k}}{(k)_k}\, \Bigl\langle F\Bigl(\eta-\sum_{q=0}^{k-1}\eta_q\Bigr)f_{m_{j,k}},f_{m_{j,k}+j}\Bigr\rangle_{L^2(\partial D)}.
	\end{equation} 
	Employing \eqref{eq:Fetacomprecon}, the term involving $F(\sum_{q=0}^{k-1}\eta_q)$ can be written as
	\begin{equation*}
		\Bigl\langle F\Bigl(\sum_{q=0}^{k-1}\eta_q\Bigr)f_{m_{j,k}},f_{m_{j,k}+j}\Bigr\rangle_{L^2(\partial D)} = -\frac{1}{\sqrt{\pi}}\sum_{q=0}^{k-1}c_{j,q} \frac{\sqrt{\abs{j}+2q+1}\,(k)_q}{(\abs{j}+k+q+1)(\abs{j}+k+q)_q}.
	\end{equation*}
	Combined with \eqref{eq:cjkrecon}, and making use of the simplifications
	\begin{equation*}
		\frac{(\abs{j}+2k)_k}{(k)_k} = \binom{\abs{j}+2k}{k}\qquad \text{and} \qquad \frac{(k)_q(\abs{j}+2k)_k}{(k)_k(\abs{j}+k+q)_q} = \binom{\abs{j}+2k}{k-q},
	\end{equation*}
	this concludes the proof of Theorem~\ref{thm:main1}. \hfill \qed
     
    \section{Injectivity for square-integrable perturbations} \label{sec:injectivity}
    
    In this section we present an alternative proof for Calder\'on's injectivity result, relying on polynomial approximation instead of the Fourier transform and harmonic exponentials. The result is also implied by Theorem~\ref{thm:main1}, but our aim is to give a less technical proof. To this end, we need an approximation result for polynomials that miss an arbitrary number of lowest order terms.
    
    \begin{lemma} \label{lemma:polapprox}
    	For $n\in \mathbb{N}_0$ let 
    	\begin{equation*}
    		\mathcal{P}_n = \mspan\{x^{n+2m} \mid m \in\N_0\}.
    	\end{equation*} 
    	The space $\mathcal{P}_n$ is dense in $L^2((0,1))$ for any $n\in \mathbb{N}_0$.	
    \end{lemma}
    \begin{proof}
    	Let $f\in\overline{\mathcal{P}_n}^\perp$, which means that
    	\begin{equation} \label{eq:polint}
    		\int_0^1 x^{n+2m}f(x)\,\di x = 0, \qquad m = 0,1,\dots \ .
    	\end{equation}
    	%
    	%We proceed to prove that $f = 0$ in $L^2((0,1))$.
        Using the even reflection 
    	\begin{equation*}
    		g(x) = \begin{cases}
    			x^nf(x) & \text{for } x\in(0,1), \\[1mm]
    			(-x)^nf(-x) & \text{for } x\in (-1,0),
    		\end{cases}
    	\end{equation*}
    	and \eqref{eq:polint}, it follows that
    	\begin{equation*} 
    		\int_{-1}^1 x^k g(x)\,\di x = 0, \qquad k=0,1,\dots \ .
    	\end{equation*}
    	Thus $g = 0$ in $L^2((-1,1))$, and thereby $f = 0$ in $L^2((0,1))$. 
    \end{proof}
    
    To conclude this work, we prove the injectivity of $F$ on $L^2(\Omega)$ perturbations.
    
    \begin{theorem} \label{thm:injective}
    	$F\in\mathscr{L}(L^2(\Omega),\mathscr{L}(L_\diamond^2(\partial\Omega)))$ is injective.
    \end{theorem}
    \begin{proof}
    	Recall that Proposition~\ref{prop:L2} yields $F\in\mathscr{L}(L^2(\Omega),\mathscr{L}(L_\diamond^2(\partial\Omega)))$. Moreover, it is sufficient to prove the injectivity for the unit disk $D$, as the corresponding result for a general simply connected $C^{1,\alpha}$ domain $\Omega$ subsequently follows from \eqref{eq:transformed}.
    	
    	Working with the polar coordinates $x = r\e^{\I\theta}$, let $\{g_m\}_{m=1}^\infty$ be an orthonormal basis for $L^2_r((0,1))$ and let $f_j$ be as defined in \eqref{eq:fj}. The functions $g_m(r)f_j(e^{\I\theta})$ for $m\in\mathbb{N}$ and $j\in\mathbb{Z}$ form an orthonormal basis for $L^2(D)$. Thus for any $\eta\in L^2(D)$ there is an $\ell^2(\Z\times\N)$ sequence of coefficients $c_{j,m}$ such that
    	\begin{equation*}
    		\eta(r\e^{\I \theta}) = \sum_{j\in\Z}\underbrace{\sum_{m\in\N}c_{j,m}g_m(r)}_{\rho_j(r)} \! f_j(e^{\I\theta}) = \sum_{j\in\Z}\rho_j(r)f_j(e^{\I\theta}).
    	\end{equation*}
    	Notice that $\rho_j\in L_r^2((0,1))$ and
    	\begin{equation} \label{eq:normeta}
    		\norm{\eta}^2_{L^2(D)} = \sum_{j\in\Z}\norm{\rho_j}^2_{L^2_r((0,1))}.
    	\end{equation}
    	We assume that $\eta \neq 0$ and proceed to prove that $F$ is injective by showing that also $F\eta \neq 0$. 
    	
    	Due to \eqref{eq:normeta}, at least one of the radial components is nonzero, say, $\rho_k$ for some $k\in\mathbb{Z}$. Since $r\rho_k$ belongs to $L^2((0,1))$, Lemma~\ref{lemma:polapprox} guarantees the existence of an arbitrarily large $n_0\in\mathbb{N}_0$ such that
    	\begin{equation*}
    		\int_0^1 \rho_k(r)\, r^{n_0+1}\,\di r \neq 0.
    	\end{equation*}
    	Moreover, we have the freedom to choose if $n_0$ is even or odd.
    	
    	Let $mn>0$. Inserting \eqref{eq:Fip} and the series representation of $\eta$ into \eqref{eq:F} gives
    	\begin{align*}
    		\inner{(F\eta)f_m,f_n}_{L^2(\partial D)} &= \frac{-1}{\pi\sqrt{2\pi}}\int_0^1\int_0^{2\pi} \sum_{j\in\mathbb{Z}}\rho_j(r)\,r^{\abs{m}+\abs{n}-1}\e^{\I(j+m-n)\theta}\,\di \theta\,\di r \\
    		&= -\sqrt{\frac{2}{\pi}}\int_0^1 \rho_{n-m}(r)\,r^{\abs{m}+\abs{n}-1}\,\di r.
    	\end{align*}
    	Let us focus on some specific indices, namely 
    	\begin{equation*}
    		n_k = \frac{1}{2}(k+n_0+2) \qquad \text{and} \qquad m_k = \frac{1}{2}(-k+n_0+2).
    	\end{equation*}
    	We can pick $n_0$ large enough so that $n_k=\abs{n_k}$ and $m_k=\abs{m_k}$. Moreover, we pick $n_0$ to be even if $k$ is even and odd if $k$ is odd, ensuring that $n_k$ and $m_k$ are integers. Then we have
    	\begin{equation*}
    		\inner{(F\eta)f_{m_k},f_{n_k}}_{L^2(\partial D)} = -\sqrt{\frac{2}{\pi}}\int_0^1 \rho_{k}(r)\, r^{n_0+1}\,\di r \neq 0,
    	\end{equation*}
    	i.e.~$F\eta \neq 0$ and $F$ is injective. 
    \end{proof} 
    
    \subsection*{Acknowledgments}
    
    The authors thank Arne Jensen (Aalborg University) for useful discussions. This work is supported by the Academy of Finland (decision 336789) and the Aalto Science Institute (AScI). HG is supported by grant 10.46540/3120-00003B from Independent Research Fund Denmark \textbar\ Natural Sciences and by The Research Foundation of DPhil Ragna Rask-Nielsen. NH is supported by Jane and Aatos Erkko Foundation via the project Electrical impedance tomography --- a novel method for improved diagnostics of stroke. 
    
	\bibliographystyle{plain}

\end{document}